\documentclass[reqno]{amsart}
\usepackage{setspace,amssymb}
\usepackage{ifpdf}
\ifpdf
 \usepackage[hyperindex,pagebackref]{hyperref}%
\else
 \expandafter\ifx\csname dvipdfm\endcsname\relax
 \usepackage[hypertex,hyperindex,pagebackref]{hyperref}
 \else
 \usepackage[dvipdfm,hyperindex,pagebackref]{hyperref}
 \fi
\fi
\allowdisplaybreaks[4]
\numberwithin{equation}{section}
\theoremstyle{plain}
\newtheorem{thm}{Theorem}[section]
\theoremstyle{remark}
\newtheorem{rem}{Remark}[section]
\DeclareMathOperator{\td}{d}

\begin{document}

\title[Integral representations and properties of Lah numbers]
{Some integral representations and properties of Lah numbers}

\author[B.-N. Guo]{Bai-Ni Guo}
\address[Guo]{School of Mathematics and Informatics, Henan Polytechnic University, Jiaozuo City, Henan Province, 454010, China}
\email{\href{mailto: B.-N. Guo <bai.ni.guo@gmail.com>}{bai.ni.guo@gmail.com},
\href{mailto: B.-N. Guo <bai.ni.guo@hotmail.com>}{bai.ni.guo@hotmail.com}}
\urladdr{\url{https://www.researchgate.net/profile/Bai-Ni_Guo/}}

\author[F. Qi]{Feng Qi}
\address[Qi]{College of Mathematics, Inner Mongolia University for Nationalities, Tongliao City, Inner Mongolia Autonomous Region, 028043, China}
\email{\href{mailto: F. Qi <qifeng618@gmail.com>}{qifeng618@gmail.com}, \href{mailto: F. Qi <qifeng618@hotmail.com>}{qifeng618@hotmail.com}, \href{mailto: F. Qi <qifeng618@qq.com>}{qifeng618@qq.com}}
\urladdr{\url{http://qifeng618.wordpress.com}}

\begin{abstract}
In the paper, the authors find some integral representations and discover some properties of Lah numbers.
\end{abstract}



\thanks{This paper was typeset using \AmS-\LaTeX}

\maketitle

\section{Introduction}

In combinatorics, Lah numbers, discovered by Ivo Lah in 1955 and usually denoted by $L(n,k)$, count the number of ways a set of $n$ elements can be partitioned into $k$ nonempty linearly ordered subsets and have an explicit formula
\begin{equation}\label{a-i-k-dfn}
L(n,k)=\binom{n-1}{k-1}\frac{n!}{k!}.
\end{equation}
Lah numbers $L(n,k)$ may also be interpreted as coefficients expressing rising factorials $(x)_n$ in terms of falling factorials $\langle x\rangle_n$, where
\begin{equation}
(x)_n=
\begin{cases}
x(x+1)(x+2)\dotsm(x+n-1), & n\ge1,\\
1, & n=0
\end{cases}
\end{equation}
and
\begin{equation}
\langle x\rangle_n=
\begin{cases}
x(x-1)(x-2)\dotsm(x-n+1), & n\ge1,\\
1,& n=0.
\end{cases}
\end{equation}
Lah numbers $L(n,k)$ may be generated by
\begin{equation}\label{Lah-generat-function-non-exp}
\frac1{k!}\biggl(\frac{x}{1-x}\biggr)^k=\sum_{n=0}^\infty L(n,k)\frac{x^n}{n!}.
\end{equation}
For more information on Lah numbers $L(n,k)$, please refer to~\cite[p.~156]{Comtet-Combinatorics-74}.
\par
In the theory of special functions, it is well known that the modified Bessel function of the first kind  $I_\nu(z)$ may be defined~\cite[p.~375, 9.6.10]{abram} by
\begin{equation}\label{I=nu(z)-eq}
I_\nu(z)= \sum_{k=0}^\infty\frac1{k!\Gamma(\nu+k+1)}\biggl(\frac{z}2\biggr)^{2k+\nu}
\end{equation}
for $\nu\in\mathbb{R}$ and $z\in\mathbb{C}$, where $\Gamma$ represents the classical Euler gamma function which may be defined~\cite[p.~255]{abram} by
\begin{equation}\label{gamma-dfn}
\Gamma(z)=\int^\infty_0t^{z-1} e^{-t}\td t
\end{equation}
for $\Re z>0$.
\par
In this paper, we will find some integral representations and properties of Lah numbers $L(n,k)$.

\section{Integral representations of Lah numbers}

We first establish integral representations of Lah numbers $L(n,k)$, in which the exponential function $e^{-1/x}$ and the modified Bessel function of the first kind  $I_1$ is involved.

\begin{thm}\label{lah-number-int-thm}
For $1\le m\le n$ and $x>0$, we have
\begin{equation}\label{lah-number-int-total-eq}
\sum_{k=1}^{n}L(n,k)x^{k}
=\frac{e^{-x}}{x^n}\int_0^\infty I_1\bigl(2\sqrt{t}\,\bigr)t^{n-1/2}e^{-t/x}\td t
\end{equation}
and
\begin{equation}\label{lah-number-int-eq}
L(n,m)=\frac1{m!}\lim_{x\to0^+}\int_0^\infty I_1\bigl(2\sqrt{t}\,\bigr)t^{n-1/2} \frac{\td^m}{\td x^m} \biggl(\frac{e^{-x-t/x}}{x^n}\biggr)\td t.
\end{equation}
\end{thm}

\begin{proof}
In~\cite[Theorem~1.2]{simp-exp-degree-revised.tex}, among other things, it was obtained that the function
\begin{equation}\label{exp=k=sum-eq-degree=k+1}
H_k(z)=e^{1/z}-\sum_{m=0}^k\frac{1}{m!}\frac1{z^m}
\end{equation}
for $k\in\{0\}\cup\mathbb{N}$ and $z\ne0$ has the integral representation
\begin{equation}\label{exp=k=degree=k+1-int}
H_k(z)=\frac1{k!(k+1)!}\int_0^\infty {}_1F_2(1;k+1,k+2;t)t^k e^{-zt}\td t
\end{equation}
for $\Re(z)>0$, where ${}_pF_q(a_1,\dotsc,a_p;b_1,\dotsc,b_q;x)$ stands for the generalized hypergeometric series which may be defined by
\begin{equation}\label{hypergeom-f}
{}_pF_q(a_1,\dotsc,a_p;b_1,\dotsc,b_q;x)=\sum_{n=0}^\infty\frac{(a_1)_n\dotsm(a_p)_n} {(b_1)_n\dotsm(b_q)_n}\frac{x^n}{n!}
\end{equation}
for complex numbers $a_i$ and $b_i\notin\{0,-1,-2,\dotsc\}$ and for positive integers $p,q\in\mathbb{N}$.
See also~\cite[Section~1.2]{Bessel-ineq-Dgree-CM.tex} and~\cite[Lemma~2.1]{QiBerg.tex}. When $k=0$, the integral representation~\eqref{exp=k=degree=k+1-int} becomes
\begin{equation}\label{open-answer-1}
e^{1/z}=1+\int_0^\infty \frac{I_1\bigl(2\sqrt{t}\,\bigr)}{\sqrt{t}\,} e^{-zt}\td t
\end{equation}
for $\Re(z)>0$. By the way, the integral representation~\eqref{open-answer-1} has been applied in~\cite{Bell-Stirling-Hypergeometric.tex}. Hence, for $n\in\mathbb{N}$ and $x>0$, we have
\begin{equation}\label{open-answer-deriv}
\bigl(e^{1/x}\bigr)^{(n)}=(-1)^n\int_0^\infty I_1\bigl(2\sqrt{t}\,\bigr)t^{n-1/2} e^{-xt}\td t.
\end{equation}
\par
In~\cite[Theorem~2]{exp-reciprocal-cm.tex} and its formally published paper~\cite[Theorem~2.2]{exp-reciprocal-cm-IJOPCM.tex}, the following explicit formula for computing the $n$-th derivative of the exponential function $e^{\pm1/x}$ was inductively obtained:
\begin{equation}\label{exp-frac1x-expans}
\bigl(e^{\pm1/x}\bigr)^{(n)}
=(-1)^n{e^{\pm1/x}}\sum_{k=1}^{n}(\pm1)^{k}L(n,k)\frac1{x^{n+k}}.
\end{equation}
By the way, the formula~\eqref{exp-frac1x-expans} have been applied in~\cite{notes-Stirl-No-JNT-rev.tex, Bell-Stirling-Lah.tex, Lah-No-Identity.tex, Filomat-36-73-1.tex, Bessel-ineq-Dgree-CM.tex, QiBerg.tex, simp-exp-degree-revised.tex}. Combining~\eqref{open-answer-deriv} and~\eqref{exp-frac1x-expans} and rearranging yield
\begin{align*}
{e^{1/x}}\sum_{k=1}^{n}L(n,k)\frac1{x^{n+k}}
&=\int_0^\infty I_1\bigl(2\sqrt{t}\,\bigr)t^{n-1/2} e^{-xt}\td t,\\
\sum_{k=1}^{n}L(n,k)\frac1{x^{k}}
&=\int_0^\infty I_1\bigl(2\sqrt{t}\,\bigr)t^{n-1/2} x^{n}e^{-xt-1/x}\td t,
\end{align*}
which may be rewritten as~\eqref{lah-number-int-total-eq}.
\par
Differentiating $1\le m\le n$ times on both sides of~\eqref{lah-number-int-total-eq} results in
\begin{equation*}
\sum_{k=m}^{n}L(n,k)\frac{k!}{(k-m)!}x^{k-m}
=\int_0^\infty I_1\bigl(2\sqrt{t}\,\bigr)t^{n-1/2} \frac{\td^m}{\td x^m} \biggl(\frac{e^{-x-t/x}}{x^n}\biggr)\td t.
\end{equation*}
Letting $x\to0^+$ in the above equation leads to~\eqref{lah-number-int-eq}. The proof of Theorem~\ref{lah-number-int-thm} is complete.
\end{proof}

\section{Properties of Lah numbers}

An infinitely differentiable function $f$ on an interval $I$ is called absolutely convex on $I$ if $f^{(2k)}(x)\ge0$ on $I$. See either~\cite[p.~375, Definition~3]{mpf-1993}, or~\cite[p.~2731, Definition~4.5]{97PA127.tex}, or~\cite[p.~617, Definiton~3]{jmaa-ii-97}, or~\cite[p.~3356, Definition~3]{AMSPROC.TEX}. A sequence $\{\mu_n\}_0^\infty$ is said to be absolutely convex if its elements are non-negative and its successive differences satisfy
\begin{equation}
\Delta^{2k}\mu_n\ge0
\end{equation}
for $n,k\ge0$, where
\begin{equation}
\Delta^k\mu_n=\sum_{m=0}^k(-1)^m\binom{k}{m}\mu_{n+k-m}.
\end{equation}

\begin{thm}\label{total-sum-lah-thm}
For $n\in\mathbb{N}$, the total sum of Lah numbers
\begin{equation}
\mathcal{L}_n=\sum_{k=1}^{n}L(n,k)
\end{equation}
is an absolutely convex sequence. Specially, the sequence $\mathcal{L}_n$ is convex.
\end{thm}

\begin{proof}
Letting $x=1$ in~\eqref{lah-number-int-total-eq} gives
\begin{equation*}
\mathcal{L}_n
=\int_0^\infty I_1\bigl(2\sqrt{t}\,\bigr)t^{n-1/2}{e^{-(1+t)}}\td t.
\end{equation*}
It is clear that the function $t^{x}$ for $t>0$ satisfies $\frac{\td^kt^x}{\td x^k}=t^x(\ln t)^k$. As a result, when $t>0$, the function $t^x$ is absolutely convex with respect to $x$. Consequently, the sequence $t^n$ is absolutely convex. Hence, the sequence $\mathcal{L}_n$ is absolutely convex. The proof of Theorem~\ref{total-sum-lah-thm} is complete.
\end{proof}

\section{A recovery of the formula~\eqref{a-i-k-dfn}}

Finally, as by-product, a recovery of the formula~\eqref{a-i-k-dfn} for Lah numbers $L(n,k)$ may be carried out as follows.
\par
The generating function~\eqref{Lah-generat-function-non-exp} may be rewritten as
\begin{equation}\label{Lah-generat-f-non-exp}
(-1)^k\frac1{k!}\biggl(\frac{x}{1+x}\biggr)^k=\sum_{n=k}^\infty(-1)^n L(n,k)\frac{x^n}{n!}.
\end{equation}
The equation~\eqref{Lah-generat-f-non-exp} may be reformulated as
\begin{equation*}
(-1)^k\frac1{k!}\frac1{(1+x)^k}=\sum_{n=k}^\infty(-1)^n L(n,k)\frac{x^{n-k}}{n!}
=\sum_{n=0}^\infty(-1)^{n+k} L(n+k,k)\frac{x^{n}}{(n+k)!}.
\end{equation*}
Because
\begin{equation}
\frac1{(1+x)^k}=\frac1{(k-1)!}\int_0^\infty t^{k-1}e^{-(1+x)t}\td t,
\end{equation}
we have
\begin{equation*}
\frac1{k!}\frac1{(k-1)!}\int_0^\infty t^{k-1}e^{-(1+x)t}\td t
=\sum_{n=0}^\infty(-1)^{n} L(n+k,k)\frac{x^{n}}{(n+k)!}.
\end{equation*}
Differentiating $m$ times with respect to $x$ on both sides of the above equation gives
\begin{equation*}
(-1)^m\frac1{k!}\frac1{(k-1)!}\int_0^\infty t^{m+k-1}e^{-(1+x)t}\td t
=\sum_{n=m}^\infty(-1)^{n} L(n+k,k)\frac{n!}{(n-m)!}\frac{x^{n-m}}{(n+k)!}.
\end{equation*}
Taking $x\to0^+$ in the above equation yields
\begin{equation*}
(-1)^m\frac1{k!}\frac1{(k-1)!}\int_0^\infty t^{m+k-1}e^{-t}\td t
=(-1)^{m} L(m+k,k)\frac{m!}{(m+k)!}
\end{equation*}
which may be rearranged as
\begin{gather*}
L(m+k,k)=\frac{(m+k)!}{m!}\frac1{k!}\frac1{(k-1)!}\int_0^\infty t^{m+k-1}e^{-t}\td t\\
=\frac{(m+k)!}{m!}\frac1{k!}\frac{(m+k-1)!}{(k-1)!}
=\frac{(m+k)!}{k!}\binom{m+k-1}{k-1}.
\end{gather*}
The formula~\eqref{a-i-k-dfn} is thus recovered.

\section{Remarks}

\begin{rem}
In the early morning of 30 December 2013, the second author searched out the paper~\cite{Lah-Exp-4-Person} in which the formula~\eqref{exp-frac1x-expans} was also found independently by five approaches. The motivation of the paper~\cite{Lah-Exp-4-Person} is different from the ones of~\cite{exp-reciprocal-cm-IJOPCM.tex} and its preprint~\cite{exp-reciprocal-cm.tex}. The motivations of the formula~\eqref{exp-frac1x-expans} in~\cite{exp-reciprocal-cm.tex, exp-reciprocal-cm-IJOPCM.tex} essentially originated from the articles~\cite{Yang-Fan-2008-Dec-simp.tex, property-psi.tex, property-psi-ii-Munich.tex} and their preprints~\cite{property-psi-ii.tex, Yang-Fan-2008-Dec.tex, property-psi.tex-arXiv}. For more information, please refer to the  expository and survey articles~\cite{bounds-two-gammas.tex, Wendel2Elezovic.tex-JIA, Wendel-Gautschi-type-ineq-Banach.tex} and their preprints~\cite{bounds-two-gammas.tex-RGMIA, Wendel-Gautschi-type-ineq.tex-arXiv, Wendel2Elezovic.tex}.
\end{rem}

\begin{rem}
In an-email bearing the date March 7, 2014, Dr. Miloud Mihoubi in Algeria recommended the paper~\cite{Ahuja-Enneking-Fibonacci-1979} to the authors. In~\cite[Lemma]{Ahuja-Enneking-Fibonacci-1979}, it was stated that if
\begin{equation}
P_{m,k}(x)=\sum_{n=1}^mL_k(m,n)x^n,
\end{equation}
then the $m$ roots of $P_{m,k}(x)$ are real, distinct, and non-positive for all $m\in\mathbb{N}$, where the associated Lah numbers $L_k(m,n)$ for $k>0$ may be defined by
\begin{equation}
L_k(m,n)=\frac{m!}{n!}\sum_{r=1}^n(-1)^{n-r}\binom{n}{r}\binom{m+rk-1}{m}
\end{equation}
and $L_k(m,n)=0$ for $n>m$. Since $L_1(m,n)=L(m,n)$, see~\cite[p.~158, Eq.~(4)]{Ahuja-Enneking-Fibonacci-1979}, when $k=1$, the polynomial $P_{m,k}(x)$ becomes
\begin{equation}
P_{m,1}(x)=\sum_{n=1}^mL(m,n)x^n.
\end{equation}
\par
For $n\in\mathbb{N}$, the integer polynomial
\begin{equation}
\mathcal{L}_n(x)=\sum_{k=0}^{n}L(n+1,k+1)x^{k}
\end{equation}
of degree $n$ satisfies
\begin{equation}
\mathcal{L}_n(x)-\frac{P_{n+1,1}(x)}x=\sum_{k=0}^{n}L(n+1,k+1)x^{k} -\sum_{k=1}^{n+1}L_k(n+1,k)x^{k-1} =0.
\end{equation}
\end{rem}

\begin{rem}
This paper is a corrected version of the preprint~\cite{Lah-Number-Int-Property.tex}.
\end{rem}

\end{document}